\documentclass[letterpaper, 10 pt, conference]{ieeeconf}
\IEEEoverridecommandlockouts
\usepackage{cite}
\usepackage{amsmath,amssymb,amsfonts}

\usepackage{amsthm}
\usepackage{algorithmic}
\usepackage{import}
\usepackage{graphicx}
\usepackage{textcomp}

\usepackage{xcolor}
\usepackage{mathtools}
\let\labelindent\relax
\usepackage{enumitem}

\usepackage{lipsum}

\usepackage{tikz}
\usetikzlibrary{positioning, calc,arrows.meta}

\usepackage[absolute,overlay]{textpos} 

\usepackage{cuted}
\usepackage{dblfloatfix} 

\newcommand{\Rankf}[1]{\operatorname{rank}(#1)}

\newcommand{\Dim}[1]{\mathrm{dim}(#1)}

\newcommand{\D}{\mathrm{d}}
\newcommand{\Lie}{\mathrm{L}}

\newcommand{\X}{\mathcal{X}}

\newcommand{\TX}{\mathcal{T}(\mathcal{X})}

\newcommand{\Spanf}[1]{\operatorname{span}\{#1\}}
\newcommand{\Spand}[1]{\mathrm{span}\left\{#1\right\}}

\newcommand{\ddiff}[1]{d_{\operatorname{diff}}(#1)}
\newcommand{\idxsum}[1]{\vert #1 \vert}

\newcommand{\pad}[1]{\partial_{#1}}

\theoremstyle{plain}      
\newtheorem{theorem}{Theorem}
\newtheorem{lemma}{Lemma}
\newtheorem{corollary}{Corollary}

\theoremstyle{definition} 
\newtheorem{definition}{Definition}
\newtheorem{example}{Example}

\theoremstyle{remark}     
\newtheorem{remark}{Remark}

\begin{document}

\title{ \textbf{On the Linearization of Flat Multi-Input
Systems via Prolongations}
\thanks{This work has been submitted to the IEEE for possible publication. Copyright may be transferred without notice, after which this version may no longer be accessible.}
\thanks{This research was funded in whole, or in part, by
the Austrian Science Fund (FWF) P36473. For the purpose of
open access, the author has applied a CC BY public copyright
licence to any Author Accepted Manuscript version arising
from this submission.}}
 
\author{Georg Hartl, Conrad Gstöttner, and Markus Schöberl
\thanks{All authors are with the Institute of Control
Systems, Johannes Kepler University Linz, Altenberger
Strasse 69, 4040 Linz, Austria,
\mbox{E-mail:}~{\tt \{georg.hartl, conrad.gstoettner,
markus.schoeberl\}@jku.at}.}}
 
\maketitle

\begin{abstract}
We examine when differentially flat nonlinear control systems
with more than two inputs can be rendered static feedback
linearizable by a minimal number of
prolongations of suitably chosen inputs after
applying a static input transformation. We derive
sufficient conditions that guarantee such prolongations yield
a static feedback linearizable system. For $(x,u)$-flat two-input systems, 
prior work established precise links between the relative degrees,
the highest derivative orders occurring in the flat
parameterization, and the minimal dimension of a linearizing
dynamic extension, leading to necessary and sufficient
criteria for flatness of systems that become static feedback linearizable
after at most two prolongations of such suitably chosen
inputs. Building on the structure of the time derivatives 
of a flat output, this work extends this analysis to 
systems with three inputs.
\end{abstract}
 
 
\section{Introduction}
 
A central objective in nonlinear control theory is to
characterize systems that can be transformed into canonical
forms such as chains of integrators using suitable state and
input transformations. In this context,
\emph{differential flatness} is
a fundamental 
concept~\cite{fliess_flatness_1995,fliess_lie-backlund_1999}. 
A nonlinear control system
\begin{equation}\label{eq:f_xu_m_inputs}
    \dot x = f(x,u)
\end{equation}
with $n$ state variables $x$ and $m$ control inputs $u$ is
called differentially flat, or simply flat, if there exists
an $m$-tuple of differentially independent functions
\begin{equation}\label{eq:flat_output_intro}
    y=\varphi(x,u,u^{(1)},\ldots,u^{(\nu)}),
\end{equation}
where $u^{(\nu)}$ denotes the $\nu$-th time derivative
of $u$, such that
\begin{equation}\label{eq:flat_para_intro}
    (x,u)=F(y, y^{(1)}, \ldots, y^{(r)}).
\end{equation}
This means both $x$ and $u$ can be expressed via the
so-called \emph{flat parameterization}~\eqref{eq:flat_para_intro}
in terms of the \emph{flat output}~\eqref{eq:flat_output_intro}
and finitely many of its time derivatives. 
A system~\eqref{eq:f_xu_m_inputs} admitting
a flat output of the form $\varphi(x)$ is called
$x$-flat, and a system admitting a flat output of the
form $\varphi(x,u)$ is called \mbox{$(x,u)$-flat}.\footnote{Every
$x$-flat system is in particular $(x,u)$-flat.}
The flatness property
enables systematic solutions to generally non-trivial tasks
such as trajectory planning and
tracking~\cite{fliess_flatness_1995,fliess_lie-backlund_1999}. 

In flatness-based
tracking control design, one typically applies an endogenous
feedback that exactly linearizes the system such that the
closed-loop input-output dynamics reduce to $m$ integrator
chains~\mbox{$y_j^{(l_j)} = w_j$},
\mbox{$j=1,\ldots,m$}, with respective lengths $l_j$,
between a new input $w$ and the flat output $y$.
 
After three decades of substantial progress in flatness
theory, computing a flat output remains a challenging task
without a general solution, as highlighted, e.g.,
by~\cite{gstottner_structurally_2022,
gstottner_necessary_2023, nicolau_flatness_2017,
nicolau_normal_2019, schoberl_implicit_2014}. A well-known
subclass of flat systems with a comprehensive
solution~\cite{jakubczyk_linearization_1980} are
\emph{static feedback linearizable} (SFL) systems, i.e.,
systems that can be exactly linearized by a static state
feedback $u=\alpha(x,w)$.%
\footnote{We use \emph{SFL} as an abbreviation for
\emph{static feedback linearization} and, by extension, for
the associated notion of being \emph{statically feedback
linearizable} and the corresponding property of
\emph{static feedback linearizability}.}
The corresponding flat output is referred to as a
\emph{linearizing output}.
 
In general, flat systems admit exact linearization by
\emph{endogenous dynamic feedback}---roughly speaking, an
endogenous feedback involves only system variables, i.e.,
states, inputs, and time derivatives of inputs---which can
be interpreted as static feedback linearization of a suitable
dynamic extension. A well-studied class of endogenous dynamic
feedback laws arises from \emph{original input prolongations},
i.e., dynamic extensions obtained by finitely many derivatives
of the original control inputs. Flat systems that become SFL
after a finite number of such input prolongations have been
investigated in depth,
see~\cite{charlet_sufficient_1991, sluis_bound_1996,
battilotti_constructive_2004, franch_linearization_2005}.
Necessary and sufficient conditions for this property were
recently established in~\cite{levine_differential_2025}.
Taking a more general perspective, one may ask whether a
system can be rendered SFL by prolongations of selected
components of a suitably transformed input
\(\hat u=\phi(x,u)\). We refer to this problem simply as
\emph{static feedback linearizability via prolongations}.
 
The minimal order of a dynamic extension required to render
a system SFL with $\varphi$ as a linearizing output is called
the \emph{differential difference of}~$\varphi$, denoted
by~$d_{\mathrm{diff}}(\varphi)$. The differential difference
of a system, simply denoted by $d_{\mathrm{diff}}$, is the
smallest differential difference among all possible flat
outputs. As shown in \cite{nicolau_flatness_2017}, every flat
system with $d_{\mathrm{diff}}=1$ becomes SFL after a
one-fold prolongation of a suitably chosen input.
Control-affine systems with differential difference of one
have been fully characterized with respect to flatness
in \cite{nicolau_flatness_2017}.%
\footnote{This result extends to general nonlinear systems,
as prolonging all inputs preserves the differential difference
while yielding a control-affine system.}
 
Previous work shows that every $(x,u)$-flat system with two
inputs becomes SFL after $d_{\mathrm{diff}}$-fold
prolongations of a suitably chosen control
input \cite{gstottner_linearization_2020}, directly linking
the state dimension, relative degrees, highest derivative
orders in the flat parameterization, and the differential
difference. By additionally establishing procedures to
identify such suitable input transformations without
explicitly requiring a flat output, it has been shown that
the gap to SFL systems can be systematically reduced.
Building on these results, necessary and sufficient
conditions have been established for flatness of two-input
systems with $d_{\mathrm{diff}}\leq2$
\cite{nicolau_flatness_2016, gstottner_necessary_2023}.
Motivated by these developments, we study static feedback
linearizability via prolongations for systems with multiple
(i.e., more than two) inputs, and derive particular results for 
the $(x,u)$-flat three-input case.
Our contributions are:
\begin{enumerate}[label=\Alph*)]
    \item For multi-input systems~\eqref{eq:f_xu_m_inputs} with
    a flat output~\eqref{eq:flat_output_intro}, we derive sufficient
    conditions for \mbox{$\varphi$-SFL} via a minimal
    number of prolongations.
    \item For three-input systems with an $(x,u)$-flat output, 
    we describe the structure of the time derivatives of
    the flat-output components under suitable input
    transformations and characterize
    \mbox{$\varphi$-SFL} via minimal prolongations in
    terms of the relative degrees, the highest derivative orders
    of the flat parameterization, and the differential
    difference.
    \item Building on these results, we present a
    procedure to verify whether a given candidate triple
    \mbox{$\varphi(x,u)=(\varphi^1,
    \varphi^2, \varphi^3)$}
    is a valid flat output.
\end{enumerate}
 
Our paper is structured as follows:
Section \ref{sec:notation} introduces notation and
terminology followed by Section \ref{sec:Known_Results}
recalling key results for SFL of two-input systems. Our main
results are shown in Section~\ref{sec:main_results} 
followed by Section \ref{sec:examples}
presenting illustrative examples.
Section \ref{sec:conclusion} gives a conclusion and an
outlook. Detailed proofs are presented
in the appendix.

\section{Notation and Terminology}
\label{sec:notation}
This work utilizes tensor notation and the Einstein summation 
convention. We drop the index range when evident from the context.
Consider an $n$-dimensional smooth manifold 
$\mathcal{X}$ with local coordinates~\mbox{$x=(x^1, \ldots, x^n)$}. By $\partial_x h$, where $h$ is an
$m$-tuple of functions~\mbox{$h=(h^1, \ldots, h^m):
\mathcal{X} \rightarrow \mathbb{R}^m$},
we indicate the $m \times n$ Jacobian matrix,
and $\partial_{x^i}h^j$ represents the partial derivative of 
$h^j$ with respect to $x^i$. The notation $\D h$ represents 
the differentials~\mbox{$(\D h^1, \ldots, \D h^m)$}. 
The Lie derivative of a scalar 
function $h^j$ along a vector field $v \in \TX$, where $\TX$ represents
the tangent bundle of $\X$, is written as $\Lie_vh^j$.
We indicate time derivatives using subscripts in square brackets.
To define the $\alpha$-th time derivative of the $j$-th component
of $y$, we use $y^j_{[\alpha]}$, 
whereas~\mbox{$y_{[\alpha]} = (y^1_{[\alpha]}, \dots ,y^m_{[\alpha]})$}.
To abbreviate the notation for time derivatives of various orders within a tuple, we use multi-indices denoted by capital Latin letters.
Consider two multi-indices 
\mbox{$A = (a^1, \dots, a^m)$} and~\mbox{$B=(b^1, \dots, b^m)$}
and assume that $A \leq B$, i.e., $a^j \leq b^j$, 
for each $j\in\{1,\ldots,m\}$. 
Using $A$ and $B$, we get the abbreviations 
\mbox{$y_{[A]} = (y^1_{[a^1]}, \dots,y^m_{[a^m]})$},
\mbox{$y_{[0, A]} = ( y^1_{[0,a^1]}, \dots,y^m_{[0,a^m]} )$},
and \mbox{$y_{[A, B]} = ( y^1_{[a^1,b^1]}, \dots,y^m_{[a^m,b^m]} )$},
where \( y^j_{[a^j, b^j]} \) represents successive time derivatives 
of $y^j$ given by~\mbox{\( y^j_{[a^j, b^j]} 
= ( y^j_{[a^j]}, \dots, y^j_{[b^j]}) \)}. If~\mbox{$a^j > b^j$}, 
then \mbox{$y^j_{[a^j, b^j]}$} is empty. Multi-indices 
are added and subtracted componentwise, i.e.,
~\mbox{$A \pm B = ( a^1\pm b^1, \dots, a^m \pm b^m )$}.
Adding or subtracting an integer $c$ to a multi-index $A$ is written 
by~\mbox{$A \pm c = ( a^1\pm c, \dots, a^m \pm c )$}.
The summation over the indices is given by
$\idxsum{A} = \sum_{j=1}^m a^j$. Unless stated otherwise, 
any given transformation is considered invertible. Throughout, we assume 
that all functions, vector fields and differentials are smooth.
We consider generic points only.

\subsection{Differential Flatness}
Following the differential geometric framework of,
e.g., \cite{kolar_properties_2016}, we consider the
extended state-input manifold
\mbox{$\X \times \mathcal{U}_{[0,l_u]}$} with local
coordinates \mbox{$(x,u_{[0,l_u]})$}. The integer $l_u$ is
chosen sufficiently large so that time derivatives of
all relevant functions along trajectories
of~\eqref{eq:f_xu_m_inputs} can be computed as Lie
derivatives along the vector field
\begin{equation}\label{eq:extended_vector_field}
f_u = f^i(x,u)\partial_{x^i}
+ \sum_{\alpha=0}^{l_u-1}
u^j_{[\alpha+1]}\partial_{u^j_{[\alpha]}}.
\end{equation}

\begin{definition}\label{def:diff_flatness}
A nonlinear system~\eqref{eq:f_xu_m_inputs} is called
differentially flat if there exist $m$ smooth functions
\begin{equation}\label{eq:flat_output}
    y=\varphi(x,u_{[0,Q]}),
\end{equation}
defined on $\X \times \mathcal{U}_{[0,l_u]}$, such
that 
\begin{equation}
    \label{eq:flat_parametrization}
    \begin{alignedat}{2}
        x^i & = F_x^i(y_{[0,R-1]}),
        & \hspace{3em} i&=1,\ldots,n, \\
        u^j & = F_u^j(y_{[0,R]}),
        &  j&=1,\ldots,m,
    \end{alignedat}
\end{equation}
with smooth functions $F^i_x$ and $F^j_u$. The
multi-index \mbox{$R=(r^1,\ldots,r^m)$} denotes the highest
derivative orders of $y$ appearing
in~\eqref{eq:flat_parametrization}. The $m$-tuple
$\varphi$ is called a flat output of
system~\eqref{eq:f_xu_m_inputs}, and $Q$ specifies the
highest derivative orders of $u$
in~\eqref{eq:flat_output}.
\end{definition}

An important consequence of Definition \ref{def:diff_flatness}, 
detailed in \cite{kolar_properties_2016}, is the linear 
independence of the differentials $\D \varphi, \D \varphi_{[1]}, 
\ldots, \D \varphi_{[\beta]}$ for arbitrary differentiation 
order $\beta$. This property ensures that the parameterization 
\eqref{eq:flat_parametrization} and the multi-index $R$ 
associated with a given flat output~\eqref{eq:flat_output} 
are locally unique. Moreover, the mapping \mbox{$(F_x,F_u):
\mathbb{R}^{\idxsum{R}+m} \to \mathbb{R}^{n+m}$} possesses 
the submersion property. By inserting \mbox{$y_{[0,R]}
=\varphi_{[0,R]}$} into~\eqref{eq:flat_parametrization} and 
taking the exterior derivative, we additionally find that 
flatness implies
\begin{subequations}\label{eq:imp_exterior_deriv}
\begin{align}
    \Spand{\D x} & 
    \subset \Spand{\D \varphi_{[0,R-1]}}, 
    \label{eq:imp_exterior_deriv_x} \\
    \Spand{\D u} & 
    \subset \Spand{\D \varphi_{[0,R]}} . 
    \label{eq:imp_exterior_deriv_u}
\end{align}
\end{subequations}

Given an $(x,u)$-flat output, we 
characterize each component $\varphi^j$ by its relative degree $k^j$.
The relative degrees \mbox{$K=(k^1,\ldots,k^m)$} are defined by
\vspace{-1ex}
\begin{equation*}
    \Lie_{f_u}^{k^j-1}\varphi^j = 
    \varphi^j_{[k^j-1]}(x), 
    \quad 
    \Lie_{f_u}^{k^j}\varphi^j = \varphi^j_{[k^j]}(x,u)
\end{equation*}
using~\eqref{eq:extended_vector_field}. Note that for any 
component $\varphi^j(x,u)$ that already depends explicitly 
on some input component, the corresponding relative degree 
is given by $k^j=0$.

\subsection{Static Feedback Linearizability}
A system of the form~\eqref{eq:f_xu_m_inputs} with a flat 
output $\varphi(x)$ is 
SFL if and only if 
$\idxsum{K}=n$. Then $R$ and $K$ coincide, 
and~\eqref{eq:flat_parametrization} constitutes a 
diffeomorphism. For SFL systems, the state and input 
transformation $(\hat x, \hat u) = (\varphi_{[0,K-1]}(x), 
\varphi_{[K]}(x,u))$ maps the system into $m$ decoupled 
integrator chains with $\varphi$ being the so-called 
\emph{linearizing output}.

\subsection{Static Feedback Linearizability via Dynamic Extensions}

Every flat system can be rendered static feedback linearizable 
via dynamic extensions. For a given flat output $\varphi$,
the flat parameterization $F_x$ is a submersion, and can thus 
be extended to a diffeomorphism \mbox{$(F_x, F_z): 
\mathbb{R}^{\idxsum{R}} \to \mathbb{R}^{\idxsum{R}}$} by 
adjoining $\idxsum{R} - n$ functions \mbox{$z = F_z(y_{[0,R-1]})$}.
The quantity
\begin{equation*}
    d_{\mathrm{diff}}(\varphi) = \idxsum{R} - n = \Dim{z}    
\end{equation*}
is called the \emph{differential difference} of 
$\varphi$~\cite{gstottner_necessary_2023}. Accordingly, the 
differential difference is the minimal dimension of a 
dynamic extension that renders a flat system SFL with its 
given flat output $\varphi$ as a linearizing output.
Additionally introducing the new input $v = y_{[R]}$ yields 
the diffeomorphism \mbox{$\Psi: (x,z,v) = (F_x(y_{[0,R-1]}), 
F_z(y_{[0,R-1]}), y_{[R]})$} with its inverse $\Psi^{-1}$.
Thereby, we obtain the dynamic extension
\begin{equation*}
\dot{z} = F_{z,[1]}(y_{[0,R]}) \circ \Psi^{-1}(x,z,v), \quad 
u = F_u \circ \Psi^{-1}(x,z,v),
\end{equation*}
yielding a linear input-output behavior $y_{[R]} = v$ 
of the closed-loop system
\begin{equation*}
\dot{x} = f(x, F_u \circ \Psi^{-1}(x,z,v)), \hspace{0.5em} 
\dot{z} = F_{z,[1]}(y_{[0,R]}) \circ \Psi^{-1}(x,z,v).
\end{equation*}

\subsection{Static Feedback Linearizability via Prolongations}

A central question for flat systems is whether they can 
be exactly linearized via static feedback after prolongations 
of suitably chosen control inputs.
We capture this property in the following definition.
\begin{definition}\label{def:SFL}
    Consider a system~\eqref{eq:f_xu_m_inputs} with a 
    flat output~\eqref{eq:flat_output}. The given system is 
    \textit{$\varphi$-SFL via prolongations} if there exist an 
    invertible static input transformation
    \begin{equation}\label{eq:def_1_input_trf}
        (\hat u_1, \hat u_2) = \Phi_{\hat u}(x,u),
    \end{equation}
    where $\Dim{\hat u_1}=m_1$, and 
    a multi-index~\mbox{$D=(d^1,\ldots,d^{m_1})$}, with each 
    $d^j >0$, such that the extended system
    \begin{equation}\label{eq:f_xu_extended}
        \begin{aligned}
            \dot x = \hat{f}(x,\hat u_1, \hat u_2),\; 
            \dot {\hat u}_1 = \hat u_{1,[1]},\; \ldots \;
            \dot {\hat u}_{1,[D-1]} = \hat u_{1,[D]}
        \end{aligned}
    \end{equation}
    is SFL with the given flat output $\varphi$ being a 
    linearizing output. If $\idxsum{D} = d_{\mathrm{diff}}(\varphi)$ 
    holds additionally, then we call the given system 
    $\varphi$\textit{-SFL via minimal prolongations}.
\end{definition}

\begin{remark}\label{rem:minimal_prolongations_characterization}
    In a general dynamic extension, the variables 
    $z = F_z(y_{[0,R-1]})$ can be arbitrary functions 
    of $y_{[0,R-1]}$. For prolongations, however, these variables 
    are constrained to be derivatives of transformed inputs, i.e., 
    $z = \hat{u}_{1,[0,D-1]}$. Minimality then requires that 
    the parameterization of the prolonged inputs 
    $\hat{u}_{1,[0,D-1]} = F_{\hat u_1, [0,D]}(y_{[0,R-1]})$ still 
    depends only on $y_{[0,R-1]}$. In contrast, non-minimal 
    prolongations would imply that $\hat{u}_{1,[0,D-1]} 
    = F_{\hat{u}_1,[0,D]}(y_{[0,R-1+k]})$ with some $k > 0$, 
    thereby increasing the required derivative order of 
    the flat output to parameterize the extended system.
\end{remark}

The following example demonstrates how the choice 
of~\eqref{eq:def_1_input_trf} influences the number of prolongations
required to render a system SFL with respect to its given 
flat output $\varphi$.

\begin{example}
    Consider the system
    \begin{equation*}\label{eq:example_1}
        \begin{alignedat}{2}
            \dot x^1 & = x^3u^1 + \frac{x^1u^2}{x^3}, & \hspace{2em} 
            \dot x^4 & =  \frac{x^5}{(x^6)^2} + \frac{u^1}{x^6} 
            - \frac{x^4u^2u^3}{ x^6 }, \\[-1ex]
            \dot x^2 & = x^3 + u^1, &  
            \dot x^5 & = (x^6)^2 + x^6u^2 + \frac{x^5}{x^6}u^2u^3, \\[-1ex]
            \dot x^3 & = u^2, &  \dot x^6 & =u^2u^3 .
        \end{alignedat}
    \end{equation*}
    
    In the following, we consider the non-minimal flat output 
    \mbox{$y = \varphi(x)=(\frac{x^1}{x^3},x^2,x^4x^6)$} with 
    relative degrees \mbox{$K=(1,1,1)$}. The corresponding flat 
    parameterization has the multi-index $R=(2,3,3)$, and thus the 
    differential difference is \mbox{$d_{\mathrm{diff}}(\varphi)=2$}.
    Given the flat parameterization
    \begin{equation*}\label{eq:flat_para_exp_1}
        \begin{aligned}
            x & = F_x (y^1_{[0,1]}, y^2_{[0,2]}, y^3_{[0,2]}), \\
            u^1 &= y^1_{[1]}, \quad u^2 = y^2_{[2]} - y^1_{[2]}, \quad 
            u^3 = F_{u^3}(y^1_{[2]},y^2_{[2,3]},y^3_{[3]}),  
        \end{aligned}
    \end{equation*}
    it follows that $F_x$ extended by $u^1 = y^1_{[1]}$,
    $u^1_{[1]} = y^1_{[2]}$, and \mbox{$u^2 = y^2_{[2]} - y^1_{[2]}$},
    corresponding to a two-fold prolongation of $u^1$ and a 
    one-fold prolongation of $u^2$, becomes a diffeomorphism. 
    The extended system is then of the form~\eqref{eq:f_xu_extended} 
    with $\hat u_1=(u^1,u^2)$ and $D=(2,1)$. 
    Since \mbox{$\idxsum{D}=3>\ddiff{\varphi}$}, 
    the proposed prolongations are not of minimal order. On the other hand, 
    applying the input transformation \mbox{$\hat u^1=\varphi^2_{[1]}(x,u)=x^3+u^1$} 
    and extending $F_x$ by \mbox{$\hat u^1 = y^2_{[1]}$} and 
    $\hat u^1_{[1]} = y^2_{[2]}$ also results in a diffeomorphism.
    The extended system is then of the form~\eqref{eq:f_xu_extended} 
    with \mbox{$\hat u^1=\varphi^2_{[1]}(x,u)$} and $D=d^1=2=\ddiff{\varphi}$.
    Hence, the system is \mbox{$\varphi$-SFL} via minimal prolongations.
    Observe that \mbox{$\hat u^1 = y^2_{[1]} = y^2_{[r^2-d^1]}$},
    which, in accordance with Remark~\ref{rem:minimal_prolongations_characterization},
    ensures that a $d^1$-fold prolongation of $\hat u^1$ does not 
    increase the highest derivative order $R$.
\end{example}

\section{Known Results for Two-Input Systems}
\label{sec:Known_Results}

We now focus on $(x,u)$-flat systems with two inputs,
recalling structural results from \cite{gstottner_linearization_2020}.
Consider a nonlinear control system $\dot x = f(x,u^1,u^2)$ with 
a flat output \mbox{$y=(\varphi^1(x,u),\varphi^2(x,u))$} characterized 
by the multi-index \mbox{$R=(r^1,r^2)$} of the corresponding flat 
parametrization and its relative degrees $K=(k^1,k^2)$. By applying 
an input transformation of the form \mbox{$(\hat{u}^1, \hat{u}^2) 
= (\varphi^1_{[k^1]}(x,u), u^2)$}, relabel the input components 
if necessary, the derivatives of both flat-output components 
up to order $R$ exhibit the specific structure
\begin{equation}\label{eq:derivative_structure_two_inputs}
\arraycolsep=0pt
    \begin{array}{rclcrcl}
        y^1_{[0,k^1-1]} & = & \varphi^1_{[0,k^1-1]}(x), & \hspace{0.5em} & 
        y^2_{[0,k^2-1]} & = & \varphi^2_{[0,k^2-1]}(x), \\[1ex]
        y^1_{[k^1]} & = & \hat{u}^1, && 
        y^2_{[k^2]} & = & \varphi^2_{[k^2]}(x,\hat{u}^1), \\[-0.5ex]
        & \vdots & & & & \vdots & \\[-0.5ex]
        y^1_{[r^1-1]} & = & \hat{u}^1_{[r^1-k^1-1]}, && 
        y^2_{[r^2-1]} & = & \varphi^2_{[r^2-1]}(x, \hat{u}^1_{[0,r^2-k^2-1]}), \\[1ex]
        y^1_{[r^1]} & = & \hat{u}^1_{[r^1-k^1]}, && 
        y^2_{[r^2]} & = & \varphi^2_{[r^2]}(x, \hat{u}^1_{[0,r^2-k^2]}, \hat{u}^2) \; ,
    \end{array}
\end{equation}
revealing that $\varphi^2_{[k^2,r^2]}$ depend on $x$ and 
time derivatives of $\hat{u}^1$ up to order $r^2-k^2$,
with $\hat{u}^2$ appearing exclusively in $\varphi^2_{[r^2]}$.
As established in \cite{gstottner_linearization_2020}, 
the following identities hold:
\begin{subequations}\label{eq:index_relations}
    \begin{equation}\label{eq:diff_equality}
        r^1-k^1 = r^2-k^2 = d_{\mathrm{diff}}(\varphi) ,
    \end{equation}
    \begin{equation}\label{eq:dimension_relations}
        r^1 + k^2 = n, \quad 
        r^2 + k^1 = n, \quad 
        n - k^1 - k^2 = d_{\mathrm{diff}}(\varphi) \; .
    \end{equation}
\end{subequations}

These algebraic constraints connect the system dimension $n$,
the relative degrees $(k^1,k^2)$, the highest derivative 
orders $(r^1,r^2)$, and the differential difference 
$\ddiff{\varphi}$. Consequently, the 
structure~\eqref{eq:derivative_structure_two_inputs} establishes 
a diffeomorphism, implying that every $(x,u)$-flat two-input 
system is \mbox{$\varphi$-SFL} via minimal prolongations. Moreover, 
the structure~\eqref{eq:derivative_structure_two_inputs} together 
with the relations~\eqref{eq:index_relations} have proven valuable 
in several contexts: deriving necessary and sufficient flatness 
conditions for systems with $d_{\mathrm{diff}} \leq 2$~\cite{gstottner_necessary_2023}, 
designing flatness-based tracking controls using quasi-static feedback 
of the original state $x$~\cite{gstottner_linearization_2020,gstottner_tracking_2024},
and constructing algorithms to identify flat-output 
components~\cite{HartlTriangularFormsXFlat2025}.
Motivated by these applications, the present work is devoted to 
extending the structural results~\eqref{eq:derivative_structure_two_inputs} 
and~\eqref{eq:index_relations} to systems with more than two inputs,
focusing particularly on three-input systems.
However, as the following example shows, such a generalization is 
not straightforward \cite{gstottner_analysis_2023}.

\vspace{-0.5ex}
\begin{example}\label{ex:three_input_comparison}
    Consider the three-input systems
    \begin{equation*}\label{eq:three_input_systems}
        \begin{aligned}
            &\text{(I)} & \hspace{0em}  
            \dot x^1 &= u^1, & 
            \dot x^2 &= u^2, & 
            \dot x^3 &= x^4 + u^1, & 
            \dot x^4 &= u^3, \\[1ex]
            &\text{(II)} & \hspace{0em} 
            \dot x^1 &= u^1, & 
            \dot x^2 &= u^2, & 
            \dot x^3 &= x^4 + u^1 + u^2, & 
            \dot x^4 &= u^3.\vspace{-1ex}
        \end{aligned}
    \end{equation*}
    Both systems have $n=4$ states and the flat output 
    \mbox{$y = (x^1, x^2, x^3)$} with relative degrees $K=(1,1,1)$.
    However, for system (I) we have $R = (2,1,2)$ 
    and thus \mbox{$d_{\mathrm{diff}}(y) = 1$}, whereas for system (II) 
    we have $R = (2,2,2)$ and thus \mbox{$d_{\mathrm{diff}}(y) = 2$}.
    This demonstrates that, unlike in the two-input case,
    the differential difference is not uniquely determined by 
    the system dimension and relative degrees alone.
\end{example}
\vspace{-2ex}

\begin{figure*}[!b]
\hrulefill
\normalsize
\begin{equation}\label{eq:deriv_form_case_rank1}
    \tag{$DS_3$}
    \begin{alignedat}{3}
        y^{1}_{[0,k^1-1]} & = \varphi^{1}_{[0,k^1-1]}(x) & \hspace{1em} 
        y^{2}_{[0,k^2-1]} & = \varphi^{2}_{[0,k^2-1]}(x) & \hspace{1em} 
        y^{3}_{[0,k^3-1]} & = \varphi^{3}_{[0,k^3-1]}(x) \\[-0.3ex]
        y^{1}_{[k^1]} & = \hat u^1_{} &  
        y^{2}_{[k^2]} & 
        = \varphi^{2}_{[k^2]}( x, \hat u^1 ) & 
        y^{3}_{[k^3]} & 
        = \varphi^{3}_{[k^3]}( x, \hat u^1 ) \\[-0.6ex]
        & \hspace{0.5em} \vdots && \hspace{0.5em} \vdots 
        && \hspace{0.5em} \vdots \\[-0.6ex]
        & & 
        y^{2}_{[p^2]} & 
        = \varphi^2_{[p^2]}(x, \hat u^1_{[0,p^2-k^2]}, u^2) & 
        y^{3}_{[p^3]} & 
        = \varphi^{3}_{[p^3]}( x, \hat u^1_{[0,p^3-k^3]}, u^2) \\[-0.6ex]
        & \hspace{0.5em} \vdots && \hspace{0.5em} \vdots 
        && \hspace{0.5em} \vdots \\[-0.6ex]
        & & 
        y^{2}_{[p^2+s]} & = 
        \varphi^2_{[p^2+s]}(x, \hat u^1_{[0,p^2-k^2+s]}, u^2_{[0,s]}, u^3) & 
        y^{3}_{[p^3+s]} & = 
        \varphi^{3}_{[p^3+s]}( x, \hat u^1_{[0,p^3-k^3+s]}, u^2_{[0,s]}, u^3) \\[-0.6ex]
        & \hspace{0.5em} \vdots && \hspace{0.5em} \vdots 
        && \hspace{0.5em} \vdots \\[-0.6ex]
        y^{1}_{[r^1-1]} & = \hat u^1_{[r^1-k^1-1]} & 
        y^{2}_{[r^2-1]} & = 
        \varphi^2_{[r^2-1]}(x, \hat u^1_{[0,r^2-k^2-1]}, & 
        y^{3}_{[r^3-1]} & = 
        \varphi^{3}_{[r^3-1]}( x, \hat u^1_{[0,r^3-k^3-1]}, \\[-0.3ex]
        & & & 
        \hspace{1.5em} u^2_{[0,r^2-p^2-1]}, u^3_{[0,r^2-p^2-s-1]}) & & 
        \hspace{1.5em} u^2_{[0,r^3-p^3-1]}, u^3_{[0,r^3-p^3-s-1]} ) \\[-0.3ex] 
        y^{1}_{[r^1]} & = \hat u^1_{[r^1-k^1]} & 
        y^{2}_{[r^2]} & = 
        \varphi^2_{[r^2]}(x, \hat u^1_{[0,r^2-k^2]}, & 
        y^{3}_{[r^3]} & = 
        \varphi^{3}_{[r^3]}( x, \hat u^1_{[0,r^3-k^3]}, \\[-0.3ex]
        & & & \hspace{1.5em} u^2_{[0,r^2-p^2]}, u^3_{[0,r^2-p^2-s]}) & & 
        \hspace{1.5em} u^2_{[0,r^3-p^3]}, u^3_{[0,r^3-p^3-s]}  )  \\ 
    \end{alignedat}
\end{equation}
\end{figure*}

\section{Main Results}
\label{sec:main_results}
In this section, we first establish sufficient conditions for general 
flat multi-input systems to be \mbox{$\varphi$-SFL} via minimal 
prolongations, and then specialize to $(x,u)$-flat three-input 
systems where a more refined analysis is possible.

\subsection{SFL of Flat Multi-Input Systems}

First, we present sufficient conditions for \mbox{$\varphi$-SFL} via 
minimal prolongations applicable to systems with an arbitrary 
number of inputs.
\begin{theorem}\label{thm:suff_gen_sys_linearizable}
    A system~\eqref{eq:f_xu_m_inputs} with a flat 
    output~\eqref{eq:flat_output} is \mbox{$\varphi$-SFL} via minimal 
    prolongations if there exists a partitioning of the components of 
    the flat output \mbox{$y=(\varphi_1(x,u),\varphi_2(x,u_{[0,Q]}))$} 
    with corresponding multi-indices $K=(K_1,K_2)$ and $R=(R_1, R_2)$,
    such that the $m_1$ components $\varphi_1$ satisfy 
    \mbox{$\Rankf{\pad{u}\varphi_1(x,u)}=m_1$}, $R_1 > K_1$ and 
    \mbox{$\vert R_1 \vert - \vert K_1 \vert = \ddiff{\varphi}$}.
    Applying any invertible input transformation of the form 
    \mbox{$(\hat u_1, \hat u_2)=(\varphi_{1,[K_1]}(x,u),\Phi_{\hat u_2}(x,u))$} 
    with a subsequent prolongation of $\hat u_1$ of order $D=R_1-K_1$ yields 
    a system of the form~\eqref{eq:f_xu_extended} that is SFL.
\end{theorem}
\begin{proof}[Proof of Theorem~\ref{thm:suff_gen_sys_linearizable}]
    Consider a system~\eqref{eq:f_xu_m_inputs} with a flat 
    output~\eqref{eq:flat_output} that meets 
    Theorem~\ref{thm:suff_gen_sys_linearizable}. Apply the input 
    transformation \mbox{$(\hat u_1, \hat u_2)=(\varphi_{1,[K_1]}(x,u), u_2)$},
    allow for permutations of the input components,
    and subsequently perform an \mbox{$(R_1-K_1)$-fold} prolongation 
    of the transformed input components $\hat u_1$. Expressing the 
    states $x_{ext}=(x, \hat u_{1,[0,R_1-K_1-1]})$ and inputs 
    $u_{ext}=(\hat u_{1,[R_1-K_1]},\hat u_2)$ of the extended system 
    in terms of the flat output yields the parameterization
    \vspace{-0.5ex}
    \begin{equation}\label{eq:thm_1_prf_extended}
    \begin{aligned}
        x & = F_x(y_{[0,R-1]}), \; 
        \hat u_{1,[0,R_1-K_1]} = y_{1,[K_1,R_1]}, \\
        \hat u_2 & = F_{\hat u_2}(y_{[0,R]})\, .
        \vspace{-0.5ex}
    \end{aligned}
    \end{equation}
    Note that despite the $(R_1-K_1)$-fold prolongation of 
    the input components $\hat u_1$ in~\eqref{eq:thm_1_prf_extended},
    there still only occur derivatives of $y$ up to the orders $R$.
    Given that \mbox{$\operatorname{dim}(x_{ext})=n+\idxsum{R_1}
    -\idxsum{K_1}=\idxsum{R}$} holds by assumption, it follows 
    that~\eqref{eq:thm_1_prf_extended} is a local diffeomorphism.
\end{proof}
\vspace{-1ex}
Theorem~\ref{thm:suff_gen_sys_linearizable} shows that if a suitable 
partition of the flat-output components exists, the corresponding 
flat-output derivatives directly provide the new inputs whose 
prolongations yield a minimal extension. For $(x,u)$-flat systems, 
this yields the following derivative structure.
\vspace{-1ex}
\begin{corollary}\label{cor:deriv_structure}
    For $(x,u)$-flat systems satisfying 
    Theorem~\ref{thm:suff_gen_sys_linearizable}, the derivatives of 
    the components of the flat output up to the orders $R_1$ and $R_2$ 
    are of the form
    \vspace{-0.5ex}
    \begin{equation}
        \begin{alignedat}{2}\label{eq:derivative_structure_m_inputs_suff}
            y_{1,[0,K_1-1]} & = \varphi_{1,[0,K_1-1]}(x), & \hspace{1em}  
            y_{2,[0,K_2-1]} & =  \varphi_{2, [0,K_2-1]}(x), \\
            y_{1,[K_1]} & = \hat u_1, &   
            y_{2,[K_2]} & =  \varphi_{2,[K_2]}(x,\hat u_{1}) , \\[-1ex]
            & \hspace{0.5em} \vdots && \hspace{0.5em} \vdots \\[-1ex]
            y_{1,[R_1-1]} & = \hat u_{1,[D-1]}, & 
            y_{2,[R_2-1]} & = \varphi_{2,[R_2-1]}(x, \\
            &&& \hspace{2em} \hat u_{1,[0,D-1]}), \\
            y_{1,[R_1]} & = \hat u_{1,[D]}, & 
            y_{2,[R_2]} & = \varphi_{2,[R_2]}(x, \\
            &&& \hspace{2em} \hat u_{1,[0,D]}, \hat u_{2}),
        \end{alignedat}
        \vspace{-0.5ex}
    \end{equation}
    with $D=R_1-K_1$ satisfying $\idxsum{D}=\ddiff{\varphi}$ and 
    $\Rankf{\pad{\hat u_2}\varphi_{2,[R_2]}}=m - m_1$.
\end{corollary}
\vspace{-0.5ex}
\begin{proof}
    Corollary~\ref{cor:deriv_structure} follows directly 
    from~\eqref{eq:thm_1_prf_extended}.
\end{proof}
The derivative structure given above generalizes the two-input 
case~\eqref{eq:derivative_structure_two_inputs} to multi-input 
systems for suitable partitions of the flat-output components.
Note that not all components of $\hat u_1$ need to appear 
explicitly in $\varphi_{2,[K_2]}$. However, since $F_{\hat u_2}$ 
must depend explicitly on $y_{[R]}$ by~\eqref{eq:thm_1_prf_extended}, 
all of $\hat u_{1,[0,D]}$ must occur among the arguments of 
$\varphi_{2,[K_2,R_2]}$. 

\subsection{SFL of $(x,u)$-Flat Systems with Three Inputs}
\label{sec:flat_three_input}

Having treated the general multi-input case,
we now study \mbox{$\varphi$-SFL} of $(x,u)$-flat systems of 
the form 
\vspace{-0.5ex}
\begin{equation}\label{eq:f_xu_3_inputs}
    \dot x = f(x,u^1, u^2, u^3)\vspace{-0.5ex}
\end{equation}
with $n$ states and three inputs. Given a flat output 
\mbox{$y=\varphi(x,u)$, let $K=(k^1,k^2,k^3)$} denote its 
relative degrees and $R=(r^1,r^2,r^3)$ the highest derivative orders 
appearing in the corresponding flat parameterization. Since $\idxsum{K}=n$ 
would imply that the system is already \mbox{$\varphi$-SFL}, we assume 
$\idxsum{K}<n$ and $1 \leq \Rankf{\pad{u}\varphi_{[K]}}\leq 2$. The latter 
follows from the definition of flatness, specifically 
from~\eqref{eq:imp_exterior_deriv}.

\begin{theorem}\label{thm:rank1_case}
    Consider a system~\eqref{eq:f_xu_3_inputs} with a 
    flat output \mbox{$y=\varphi(x,u)$}. After a possible 
    relabeling of the inputs, let  
    \mbox{$\hat u^1 = \varphi^1_{[k^1]}(x,u)$} replace $u^1$.
    Then the following hold:
    
    \hspace{-1em} a) The derivatives $\varphi_{[0,R]}$
    take the form~\eqref{eq:deriv_form_case_rank1}.
    Here, after a possible relabeling of $u^2$ and $u^3$,
    $p^2$ and $p^3$ denote the smallest derivative orders
    at which $\varphi^j_{[p^j]}$, \mbox{$j\in\{2,3\}$},
    explicitly depend on $u^2$, and $s \geq 0$ is the
    smallest integer such that $\varphi^j_{[p^j+s]}$,
    $j\in\{2,3\}$, explicitly depend on $u^3$.
    
    \hspace{-1em} b) For the multi-indices $K$, $R$, and $(p^2, p^3)$ 
    it follows that
    \begin{enumerate}[label=\roman*)]
        \item among the three differences $r^i-k^i$, 
        $i \in \{1,2,3\}$, the two largest always coincide.
        
        \item $r^2-p^2 = r^3-p^3$.
        
        \item the differential difference and $\Dim{x}=n$ are given by
        \begin{equation*}
            d_{\mathrm{diff}}(\varphi) = (r^1-k^1)+(r^2-p^2), \text{ and } n = k^1 + p^2 + r^3 \, .
        \end{equation*}
    \end{enumerate}

    \hspace{-1em} c) According to item b)i), the flat-output components 
    can always be relabeled such that
    \begin{equation}\label{eq:rearrangement_case_rank1}
        r^1-k^1 = r^3-k^3 \geq r^2-k^2.
    \end{equation}
    Given any arrangement satisfying
    \eqref{eq:rearrangement_case_rank1}, the system is
    \mbox{$\varphi$-SFL} via minimal prolongations with
    \begin{equation}\label{eq:input_trf_case_rank1}
        (\hat u^1,\hat u^2,\hat u^3)=
        \bigl(\varphi^1_{[k^1]}(x,u),
        \phi_{\hat u^2}(x,u),
        \phi_{\hat u^3}(x,u)\bigr)
    \end{equation}
    if and only if one of the following conditions holds:
    \begin{enumerate}[label=\roman*)]
        \item $r^1-k^1 = d_{\mathrm{diff}}(\varphi)$,
        in which case $\hat u_1 = (\hat u^1)$ and
        \mbox{$D = (r^1-k^1)$} for any invertible
        transformation~\eqref{eq:input_trf_case_rank1}.    
        
        \item $r^1-k^1 < d_{\mathrm{diff}}(\varphi)$
        and there exists a transformation
        \mbox{$\hat u^2 = \phi_{\hat u^2}(x, u)$} such
        that, after
        applying~\eqref{eq:input_trf_case_rank1}, the
        derivatives $\varphi_{[0,R]}$ take the
        form~\eqref{eq:deriv_form_case_rank1} with
        $u^2, u^3$ replaced by $\hat u^2, \hat u^3$ and
        with $s = r^2-p^2$, i.e., $\hat u^3$ appears
        exclusively in $\varphi^2_{[r^2]}$ and/or $\varphi^3_{[r^3]}$. 
        Then the system is
        \mbox{$\varphi$-SFL} with
        $\hat u_1 = (\hat u^1, \hat u^2)$ and
        \mbox{$D = (r^1-k^1,\, r^2-p^2)$}.
    \end{enumerate}
\end{theorem}

The proof of Theorem~\ref{thm:rank1_case} is given in
Appendix~\ref{appendix}.
Theorem~\ref{thm:rank1_case} holds regardless of
whether $\Rankf{\pad{u}\varphi_{[K]}}=1$ or
$\Rankf{\pad{u}\varphi_{[K]}}=2$.
The following two
corollaries address each case separately. For rank one,
Corollary~\ref{cor:rearrangement} clarifies that the
input transformation~\eqref{eq:input_trf_case_rank1} is
only meaningful when the flat-output components are
arranged according
to~\eqref{eq:rearrangement_case_rank1}.


\begin{corollary}\label{cor:rearrangement}
    Consider a system~\eqref{eq:f_xu_3_inputs} with a flat output 
    $y=\varphi(x,u)$ satisfying $\Rankf{\pad{u}\varphi_{[K]}}=1$.
    If the flat output components are arranged such that 
    $r^1-k^1 < r^2-k^2 = r^3-k^3$, i.e., such 
    that~\eqref{eq:rearrangement_case_rank1} is not met, then there 
    exists no input transformation of the 
    form~\eqref{eq:input_trf_case_rank1} such that the given system 
    can be rendered \mbox{$\varphi$-SFL} via minimal prolongations 
    of the newly chosen inputs.
\end{corollary}

\begin{proof}[Proof of Corollary~\ref{cor:rearrangement}]
    Independently of the differences \mbox{$r^i-k^i$},
    after applying the input transformation
    $\hat u^1 = \varphi^1_{[k^1]}$, the time derivatives
    $\varphi_{[0,R]}$ take the
    form~\eqref{eq:deriv_form_case_rank1}. Rendering the
    system \mbox{$\varphi$-SFL} via minimal prolongations
    requires extending the state $x$ by time derivatives
    of suitably transformed inputs $\hat u^1, \hat u^2$
    such that the map
    \begin{equation}\label{eq:map_prf_cor}
        (x, \hat u^1_{[0,d^1-1]},
        \hat u^2_{[0,d^2-1]}) \mapsto \varphi_{[0,R-1]}
    \end{equation}
    becomes a diffeomorphism, i.e., such that the number
    of variables matches on both sides. The assumption
    \mbox{$r^1-k^1 < r^2-k^2 = r^3-k^3$} implies that
    the functions $\varphi^j_{[k^j,r^j]}$,
    $j\in\{2,3\}$, depend explicitly on
    $\hat u^1_{[0,r^2-k^2]}$. Thus, $\hat u^1$ must be
    prolonged by at least $r^2-k^2 > r^1-k^1$, which
    implies $r^1-k^1 < \ddiff{\varphi}$ and consequently
    $r^2-p^2 > 0$ by \mbox{$d_{\mathrm{diff}}(\varphi) 
    = (r^1-k^1)+(r^2-p^2)$}.
    Moreover, there must exist a suitable transformation
    $\hat u^2 = \phi_{\hat u^2}(x, \hat u^1, u^2, u^3)$
    as described in item~c)ii) of
    Theorem~\ref{thm:rank1_case}, implying
    $d^2 = r^2-p^2$. Hence,
    \[
        n + (r^2-k^2) + (r^2-p^2)
        > n + (r^1-k^1) + (r^2-p^2) = \idxsum{R}
    \]
    holds for the total number of variables and thereby
    contradicts the requirement
    that~\eqref{eq:map_prf_cor} is a diffeomorphism.
\end{proof}

For rank two, the following corollary shows that
\mbox{$\varphi$-SFL} via minimal prolongations is always
achievable, with the input transformation given directly
by flat-output derivatives.

\begin{corollary}\label{cor:rank2_case}
    Consider a system~\eqref{eq:f_xu_3_inputs}. Let $y=\varphi(x,u)$ 
    be a valid flat output and let its components be arranged 
    such that~\eqref{eq:rearrangement_case_rank1} holds.
    If $\Rankf{\pad{u}\varphi_{[K]}} = 2$, 
    then the following properties are satisfied:
    
    a) The system and $\varphi$ satisfy
    the conditions of Theorem~\ref{thm:suff_gen_sys_linearizable} with
    $\varphi_1=(\varphi^1, \varphi^2)$,
    $\varphi_2=\varphi^3$,
    $\Rankf{\pad{u}\varphi_{1,[K_1]}(x,u)}=2$ and 
    $(r^1-k^1) + (r^2-k^2) = \ddiff{\varphi}$.

    b) After applying any invertible input transformation of the form 
    \mbox{$(\hat u^1, \hat u^2, \hat u^3)=(\varphi^1_{[k^1]}(x,u),
    \varphi^2_{[k^2]}(x,u),
    \phi_{\hat u^3}(x,u))$} 
    the flat-output derivatives take the form 
    \eqref{eq:derivative_structure_m_inputs_suff} with 
    $\hat u_1 = (\hat u^1, \hat u^2)$, $\hat u_2 = \hat u^3$, and
    $D=(r^1-k^1, r^2 - k^2)$.
    
    c) A prolongation of $\hat u_1=(\hat u^1, \hat u^2)$ 
    of order \mbox{$D=(r^1-k^1, r^2 - k^2)$} yields 
    a system~\eqref{eq:f_xu_extended} that is $\varphi$-SFL.
    
\end{corollary}
\begin{proof}[Proof of Corollary~\ref{cor:rank2_case}]
    Consider an $(x,u)$-flat three-input
    system~\eqref{eq:f_xu_3_inputs} with 
    $\Rankf{\pad{u}\varphi_{[K]}} = 2$ and 
    where~\eqref{eq:rearrangement_case_rank1} holds and 
    the input transformation 
    $\hat u^1 = \varphi^1_{[k^1]}$ has
    been applied. Using Theorem~\ref{thm:rank1_case} 
    item \emph{b)ii)},
    we distinguish two cases.

    \emph{Case $r^2-k^2=r^3-k^3$:} 
    Since $p^2=k^2$ and $p^3=k^3$, any two flat-output
    components can be chosen as $\varphi_1$, and
    $\Rankf{\pad{u}\varphi_{[K]}} = 2$ directly yields
    $\Rankf{\pad{u}\varphi_{1,[K_1]}} = 2$.

    \emph{Case $r^2-k^2 < r^3-k^3$:} 
    With $p^2-k^2 < p^3-k^3$, it follows that
    $p^2=k^2$ and thus
    $\Rankf{\pad{u}(\varphi^1_{[k^1]},
    \varphi^2_{[k^2]})} = 2$.

    In both cases, the corollary follows from
    Theorem~\ref{thm:suff_gen_sys_linearizable} and 
    item~\textit{b)} of Theorem~\ref{thm:rank1_case}.
\end{proof}

\subsection{Verification of a Flat-Output Candidate}
\label{sec:identify_flat_output}

This section demonstrates how the structural results of
Theorem~\ref{thm:rank1_case} can be used to verify
flat-output candidates. To this end, we first establish
how the multi-index $R$ can be determined from $K$,
$p^2$, and $p^3$ alone.

\begin{lemma}\label{lem:multi-index_R}
    Consider a system~\eqref{eq:f_xu_3_inputs} with a
    flat output $\varphi(x,u)$ whose derivatives
    $\varphi_{[0,R]}$ take the
    form~\eqref{eq:deriv_form_case_rank1}. If the
    flat-output components are arranged such
    that~\eqref{eq:rearrangement_case_rank1} holds, then
    the multi-index $R$ is given by
    \begin{equation}\label{eq:R_from_KP}
        R = (n\!-\!k^3\!-\!p^2,\;
        n\!-\!k^1\!-\!p^3,\;
        n\!-\!k^1\!-\!p^2).
    \end{equation}
\end{lemma}

\begin{proof}
    Combining items~\textit{b)ii)} and~\textit{b)iii)} of
    Theorem~\ref{thm:rank1_case} yields
    \begin{equation}\label{eq:n_eq_kpr}
        n = k^1 + p^2 + r^3 = k^1 + p^3 + r^2 \, ,
    \end{equation}
    from which $r^2 = n - k^1 - p^3$ and
    $r^3 = n - k^1 - p^2$ follow directly. For $r^1$,
    using $\ddiff{\varphi} = \idxsum{R} - n$ together
    with item~\textit{b)iii)} gives $r^1 = n - k^3 - p^2$,
    where the arrangement~\eqref{eq:rearrangement_case_rank1},
    i.e., $r^1-k^1 = r^3-k^3$, was used.
\end{proof}

For a nonlinear
system~\eqref{eq:f_xu_m_inputs}, verifying whether a
given $m$-tuple $\varphi(x, u_{[0,\nu]})$ is a flat
output is not straightforward, since the multi-index $R$
is \textit{a priori} unknown. For $(x,u)$-flat
three-input systems of the
form~\eqref{eq:f_xu_3_inputs}, however, the derivative
structure~\eqref{eq:deriv_form_case_rank1} together with
Lemma~\ref{lem:multi-index_R} yields a systematic
verification procedure. Given a candidate triple
$\varphi(x,u)$, one proceeds as follows:%
\footnote{Steps~\ref{alg:step2} and~\ref{alg:step3}
may require a relabeling of the inputs $(u^1,u^2,u^3)$.}
\begin{enumerate}[label=\arabic*)]
    \item \label{alg:step1}
    Compute the relative degrees
    $K = (k^1, k^2, k^3)$.

    \item \label{alg:step2}
    Apply the input transformation
    $\hat u^1 = \varphi^1_{[k^1]}(x,u)$ to replace
    $u^1$.

    \item \label{alg:step3}
    Differentiate $\varphi^2$ and $\varphi^3$ with
    respect to time until $u^2$ explicitly appears, thereby determining $p^2$ and $p^3$.

    \item \label{alg:step4}
    If $p^2-k^2 \neq p^3-k^3$, relabel the flat-output
    components such that $p^2-k^2 < p^3-k^3$ and
    proceed with step~\ref{alg:step7}.

    \item \label{alg:step5}
    If $p^2-k^2 = p^3-k^3$, it follows that
    \mbox{$r^2-k^2 = r^3-k^3$}. Hence, either all differences
    $r^i-k^i$ are equal, or
    \mbox{$r^1-k^1 < r^2-k^2 = r^3-k^3$}. In both cases,
    permuting $\varphi^1$ and $\varphi^2$
    ensures~\eqref{eq:rearrangement_case_rank1}.

    \item \label{alg:step6}
    For the permuted components, apply
    steps~\ref{alg:step2} and~\ref{alg:step3} with
    the original inputs $u$. Proceed
    with step~\ref{alg:step7}.

    \item \label{alg:step7}
    By Lemma~\ref{lem:multi-index_R}, the multi-index
    $R$ is given by~\eqref{eq:R_from_KP}.

    \item \label{alg:step8}
    Compute $\varphi_{[0,R-1]}$ and verify~\eqref{eq:imp_exterior_deriv_x}, 
    i.e., that the state can be reconstructed from $\varphi$ and its
    derivatives.
\end{enumerate}

\section{Examples}
\label{sec:examples}
This section illustrates Theorem~\ref{thm:rank1_case}
via an academic example followed by a physically inspired
example demonstrating Corollary~\ref{cor:rank2_case}.

\begin{example}[An academic example]
    Consider the following system 
    from~\cite[Sec. 4.4.2]{kolar_contributions_2017}:
    \stepcounter{equation}
    \begin{equation}\label{eq:exa_n_7_m_3_system}
        \begin{alignedat}{2}
            \dot x^1 & = u^1, & \hspace{1.5em} \dot x^5 & = -x^6 + x^4x^7u^1, \\
            \dot x^2 & = x^3+x^4u^1, & \dot x^6 & = -x^5u^1 + x^7(u^1u^3 -u^2-1) \\
            \dot x^3 & = u^2 - u^1u^3, & & \hspace{1em} + (x^4+u^1)x^4u^1, \\
            \dot x^4 & = u^3, & \dot x^7 & = x^4+u^1. \\
        \end{alignedat}
    \end{equation}
    System~\eqref{eq:exa_n_7_m_3_system} admits the flat output 
    $y=\varphi(x)=(x^2,x^1,x^5)$ with relative degrees $K=(1,1,1)$ and
    multi-index \mbox{$R=(4,3,4)$}, yielding \mbox{$d_{\mathrm{diff}}(\varphi)=4$}. 
    
    \emph{Item a) of Theorem~\ref{thm:rank1_case}:} 
    After applying the input transformation 
    $\hat u^1=\varphi^1_{[1]}(x,u)= x^3 + u^1x^4$, the time 
    derivatives $y_{[0,R]}=\varphi_{[0,R]}$ given by $y^1=x^2$,
    $y^1_{[1,4]}=\hat u^1_{[0,3]}$, and
    \begin{equation}\label{eq:deriv_struct_academic}
        \begin{alignedat}{2}
            y^2 &\hspace{-0.2em}=\hspace{-0.2em} x^1, & \hspace{1em}
            y^3 &\hspace{-0.2em}=\hspace{-0.2em} x^5, \\
            y^2_{[1]} &\hspace{-0.2em}=\hspace{-0.2em} 
            (\hat u^1\!-\!x^3)/x^4, & 
            y^3_{[1]} &\hspace{-0.2em}=\hspace{-0.2em} 
            (\hat u^1\!-\!x^3)x^7\!-\!x^6, \\
            y^2_{[2]} &\hspace{-0.2em}=\hspace{-0.2em} 
            (\hat u^1_{[1]}\!-\!u^2)/x^4, & 
            y^3_{[2]} &\hspace{-0.2em}=\hspace{-0.2em} 
            \varphi^3_{[2]}(x,\!\hat u^1_{[0,1]}), \\
            y^2_{[3]} &\hspace{-0.2em}=\hspace{-0.2em} 
            \varphi^2_{[3]}(x,\!\hat u^1_{[1,2]},\!u^2_{[0,1]},\!u^3), & 
            y^3_{[3]} &\hspace{-0.2em}=\hspace{-0.2em} 
            \varphi^3_{[3]}(x,\!\hat u^1_{[0,2]},\!u^2), \\
            & & 
            y^3_{[4]} &\hspace{-0.2em}=\hspace{-0.2em} 
            \varphi^3_{[4]}(x,\!\hat u^1_{[0,3]},\!u^2_{[0,1]},\!u^3)\, , 
        \end{alignedat}
    \end{equation}
    take the form~\eqref{eq:deriv_form_case_rank1} with 
    $p^2=2$, $p^3=3$, and $s=1$.
    
    \emph{Item b) of Theorem~\ref{thm:rank1_case}:}
    With $R - K = (3,2,3)$ and \mbox{$r^2-p^2 = r^3-p^3 = 1$},
    items~\textit{i)} and~\textit{ii)} are readily verified.
    Item~\textit{iii)} follows from
    \mbox{$\ddiff{\varphi}=(r^1-k^1)+(r^2-p^2)=4$}
    and \mbox{$n=k^1+p^2+r^3=7$}.
    
    \emph{Item c) of Theorem~\ref{thm:rank1_case}:} 
    The flat-output components are arranged such 
    that~\eqref{eq:rearrangement_case_rank1} holds.
    Since \mbox{$r^1-k^1=3<\ddiff{\varphi}$}, item \textit{c)ii)} applies.
    With~\eqref{eq:deriv_struct_academic} already meeting the conditions 
    of \textit{c)ii)}, it remains to apply the input 
    transformation \mbox{$(\hat u^1,\hat u^2,\hat u^3)
    =\bigl(\varphi^1_{[1]}(x,u),\,u^2,\,u^3\bigr)$}, with a subsequent
    three-fold prolongation of $\hat u^1$ and a one-fold prolongation 
    of $\hat u^2$. Then, the extended system
    \begin{equation*}
        \dot{x}=f(x,u), \quad 
        \dot{\hat u}^1_{[0,d^1-1]} = \hat u^1_{[1,d^1]}, \quad 
        \dot{\hat u}^2_{[0,d^2-1]} = \hat u^2_{[1,d^2]}
    \end{equation*}
    with $D=(r^1-k^1, r^2-p^2)=(3,1)$ satisfying 
    \mbox{$\idxsum{D}=d_{\mathrm{diff}}(\varphi)=4$} 
    is \mbox{$\varphi$-SFL}. Thus,~\eqref{eq:exa_n_7_m_3_system} is 
    \mbox{$\varphi$-SFL} via minimal prolongations.
\end{example}
\vspace{-1ex}

\begin{example}[Aerial manipulator]
    A model of the aerial manipulator as studied 
    in~\cite{HartlExactLinearizationMinimally2024a} is given by 
    the equations
    \begin{equation}\label{eq:f_plan_man}
        \dot q^i = v^i, \quad 
        \dot v^i = f^i(q,v,u), \quad i=1,\ldots,4,
    \end{equation}
    with the configuration variables $q=(x_e, z_e, \theta, \phi)$,
    the corresponding velocities 
    $v=(v_{x_e}, v_{z_e},\omega_{\theta}, \omega_{\phi})$ and the 
    control inputs $u=(u^1,u^2,u^3)$.
    As shown in \cite{welde_role_2023}, a flat output 
    of~\eqref{eq:f_plan_man} is given by 
    \mbox{$y = (\theta, \; z_e - r_e\sin(\phi+\theta), \; 
    x_e - r_e\cos(\phi+\theta))$}, where $r_e > 0$ is a constant parameter. 
    With $K=(2,2,2)$ and $R=(4,4,4)$, \eqref{eq:rearrangement_case_rank1} 
    is fulfilled.     
    Let us verify each part of Corollary~\ref{cor:rank2_case}:
    
    \emph{Item a):} With
    $\Rankf{\pad{u}(\varphi^1_{[k^1]},\varphi^2_{[k^2]})}=2$ and 
    \mbox{$(r^1-k^1) + (r^2-k^2) = 4 = \ddiff{\varphi}$},
    the conditions of Theorem~\ref{thm:suff_gen_sys_linearizable} are
    met with $\varphi_1=(\varphi^1,\varphi^2)$, $\varphi_2=\varphi^3$.
    
    \emph{Item b):} 
    Applying the regular input transformation 
    \begin{equation*}
        (\hat u^1, \hspace{-0.1em}\hat u^2, \hspace{-0.1em}\hat u^3) 
        = ( \varphi^1_{[2]}(\theta, \phi, \omega_\theta, \omega_\phi, u), 
            \varphi^2_{[2]}(\theta, \phi, \omega_\theta, \omega_\phi, u), u^3 )
    \end{equation*}
    the time derivatives of the components 
    $\varphi^1,\varphi^2$ take the form
    \begin{subequations}\label{eq:aerial_manip_deriv}
    \begin{equation}
        \begin{aligned}
            y^1 & \hspace{-0.2em}=\hspace{-0.2em} \theta, & 
            y^1_{[1]} & \hspace{-0.2em}=\hspace{-0.2em} \omega_\theta, & 
            y^1_{[2,4]} & \hspace{-0.2em}=\hspace{-0.2em} \hat u^1_{[0,2]}, \\[0.5ex]
            y^2 & \hspace{-0.2em}=\hspace{-0.2em} \varphi^2(z_e, \theta, \phi), \hspace{-0.6em} &  
            y^2_{[1]} & \hspace{-0.2em}=\hspace{-0.2em} \varphi^2_{[1]}(v_{z_e}, \theta, \phi, \omega_\theta, \omega_\phi), \hspace{-0.6em} & 
            y^2_{[2,4]} & \hspace{-0.2em}=\hspace{-0.2em} \hat u^2_{[0,2]}. \\[0.5ex]
        \end{aligned}
    \end{equation}
    and the derivatives of $\varphi^3$ are given by
    \begin{equation}
    \begin{aligned}
            y^3 & = \varphi^3(x_e, \theta, \phi), &  
            y^3_{[1]} & = \varphi^3_{[1]}(v_{x_e}, \xi), \\
            y^3_{[2]} & = \varphi^3_{[2]}(\xi, \hat u^1, \hat u^2), & 
            y^3_{[3]} & = \varphi^3_{[3]}(\xi, \hat u^1_{[0,1]}, \hat u^2_{[0,1]}), \\
            y^3_{[4]} & = \varphi^3_{[4]}(\xi, \hat u^1_{[0,2]}, \hat u^2_{[0,2]}, \hat u^3).
    \end{aligned}
    \end{equation}
    \end{subequations}
    with $\xi=(\theta, \phi, \omega_\theta, \omega_\phi)$. 
    The derivatives~\eqref{eq:aerial_manip_deriv} are of the 
    form~\eqref{eq:derivative_structure_m_inputs_suff} with 
    $\hat u_1 = (\hat u^1, \hat u^2)$, $\hat u_2 = \hat u^3$
    and $D=(2,2)$.
    
    \emph{Item c):} Since the mapping defined 
    by~\eqref{eq:aerial_manip_deriv} is a diffeomorphism, extending the 
    original state by the prolongations $\hat u^1_{[0,1]}$ and 
    $\hat u^2_{[0,1]}$ renders the system \mbox{$\varphi$-SFL}. Thus, 
    the aerial manipulator~\eqref{eq:f_plan_man} is \mbox{$\varphi$-SFL} 
    via minimal prolongations with 
    $D=(r^1-k^1, r^2-k^2)=(2,2)$.
\end{example}

\section{Conclusion}
\label{sec:conclusion}
Our work has introduced sufficient conditions for flat
multi-input systems to be \mbox{$\varphi$-SFL} via minimal
prolongations. Furthermore, for $(x,u)$-flat three-input
systems, \mbox{$\varphi$-SFL} via minimal prolongations has
been characterized in terms of $K$, $R$ and the
structure of the time derivatives of the flat output
components after applying suitable static input
transformations. Distinguishing cases according to
$\Rankf{\pad{u}\varphi_{[K]}}$, systems with rank two
can always be rendered $\varphi$-SFL via a minimal
number of prolongations (Corollary~\ref{cor:rank2_case}c)), 
while systems with rank
one require an additional structural condition
(Theorem~\ref{thm:rank1_case}c)). These findings extend
the two-input results to the three-input case and establish
explicit relationships between the state dimension, the
differential difference, $K$ and $R$. Given that
prolongations of suitably chosen inputs systematically
reduce the distance to static feedback linearizability,
future work toward solving the flatness problem for the
broader class of three-input systems could address the
following open problems: characterizing when item~\textit{c)ii)} of
Theorem~\ref{thm:rank1_case} can be satisfied, identifying
system classes that are guaranteed \mbox{$\varphi$-SFL} via
minimal prolongations, and developing constructive
procedures for computing suitable input transformations
that do not require a given flat output. Thereby,
algorithmic flatness tests become more feasible.

\bibliography{IEEEabrv,mybibfile}

\appendices

\section{ }
\label{appendix}

\begin{proof}[Proof of Theorem~\ref{thm:rank1_case}]
The proof is organized according to the items of the
Theorem.

\emph{Item a):} Since 
$1\leq\Rankf{\pad{u}\varphi_{[K]}}\leq2$, there always
exists a relabeling of the input and flat output
components, such that applying
$\hat u^1 = \varphi^1_{[k^1]}(x,u)$ replaces $u^1$.
Then the derivatives of $\varphi^1$ take the form
\mbox{$y^1_{[0,k^1-1]} = \varphi^1_{[0,k^1-1]}(x)$} and
\mbox{$y^1_{[k^1,r^1]} = \hat u^1_{[0,r^1-k^1]}$}.
Let $p^2 \geq k^2$ denote the smallest derivative order at
which $\varphi^2_{[p^2]}$ explicitly depends on $u^2$.
For the rank-two case, we can always relabel the 
flat-output components, such that $p^2=k^2$.
The integer $s \geq 0$ is the smallest
integer such that $\varphi^2_{[p^2+s]}$ and/or $\varphi^3_{[p^3+s]}$
explicitly depend on $u^3$.
Let us now consider the extended system
\begin{equation}\label{eq:ext_sys_proof}
    \dot x = f(x, \hat u^1, u^2, u^3), \hspace{1em}
    \dot{\hat u}^1_{[0,r^1-k^1-1]}
    = \hat u^1_{[1,r^1-k^1]}
\end{equation}
with the state
$x_{ext}=(x,\hat u^1_{[0,r^1-k^1-1]})$ and the input
\mbox{$u_{ext}=(\hat u^1_{[r^1-k^1]},u^2,u^3)$}. Note
that the extended system~\eqref{eq:ext_sys_proof} is still
flat with the original (minimal) multi-index $R$. Next, we
apply an input transformation
$\hat u^2 = \varphi^2_{[p^2]}$, that is static
for~\eqref{eq:ext_sys_proof}, and thereby replace $u^2$.
Consequently, the derivatives
$\varphi^1_{[0,r^1]}$ are unchanged and the remaining
time derivatives take the form
\begin{equation}\label{eq:flat_output_deriv_proof}
    \begin{alignedat}{2}
        y^2_{[0,p^2-1]} &= \varphi^2_{[0,p^2-1]}(x_{ext}), & 
        y^3_{[0,p^3-1]} &= \varphi^3_{[0,p^3-1]}(x_{ext}), \\
        y^{2}_{[p^2]} & =\hat u^2, & 
        y^{3}_{[p^3]} & = \varphi^{3}_{[p^3]}( x_{ext}, \hat u^2 ), \\[-1ex]
        & \hspace{0.5em} \vdots && \hspace{0.5em} \vdots \\[-1ex] 
        y^{2}_{[r^2-1]} & = \hat u^2_{[r^2-p^2-1]}, & \hspace{1em} 
        y^{3}_{[r^3-1]} & = \varphi^{3}_{[r^3-1]}( x_{ext}, \\
        &&& \hspace{1.5em} \hat u^2_{[0,r^3-p^3-1]}), \\
        y^{2}_{[r^2]} & = \hat u^2_{[r^2-p^2]}, & 
        y^{3}_{[r^3]} & = \varphi^{3}_{[r^3]}( x_{ext}, \hat u^1_{[r^1-k^1]}, \\
        &&& \hspace{1.5em} \hat u^2_{[0,r^3-p^3]}, u^3 ), \\
    \end{alignedat}
\end{equation}
where $p^3$ is the smallest derivative order such that 
$\varphi^3_{[p^3]}$ explicitly depends on $\hat u^2$.
The functions $\varphi^3_{[p^3,r^3-1]}$ depend solely on
$x_{ext}$ and $\hat u^2_{[0,r^3-p^3-1]}$, since $R$ is
by assumption the minimal multi-index for
which~\eqref{eq:flat_parametrization} holds.
Conversely, if the
first occurrence of $u^3$ would be in $\varphi^3_{[q]}$
for some \mbox{$p^3 \leq q \leq r^3-1$}, then the
differentials $\D\varphi^3_{[q,r^3-1]}$ contain nonzero
components in the directions $\D u^3_{[0,r^3-q-1]}$ that
do not appear in \(\D\varphi^1_{[0,r^1-1]}\) or
\(\D\varphi^2_{[0,r^2-1]}\). Hence, these components
cannot be canceled by linear combinations within
\(\Spanf{\D\varphi_{[R-1]}}\) contradicting
\mbox{$\Spanf{\D x_{ext}}
\subset \Spanf{\D \varphi_{[0,R-1]}}$}.
For the case that $\varphi^2_{[p^2+s]}$ explicitly depends on $u^3$, 
plugging \mbox{$\hat u^2 = \varphi^2_{[p^2+s]}(x_{ext},
u^2_{[0,s]}, u^3)$} into
\mbox{$\varphi^3_{[p^3+s]}(x_{ext},
\hat u^2_{[0,s]})$} implies that
$\varphi^3_{[p^3+s]}$ also explicitly depends on $u^3$.
Conversely, if none of the functions $\varphi^2_{[k^2,r^2]}$
explicitly depends on $u^3$, then, by analogous reasoning 
as above, it follows that $u^3$ must exclusively appear in $\varphi^3_{[r^3]}$.

Finally, by the chain rule we
obtain~\eqref{eq:deriv_form_case_rank1}.

For item \textit{b)ii)} and item \textit{b)iii)}, 
we consider the exterior
derivative of $\varphi^1_{[0,r^1]}=(\varphi^1_{[0,k^1-1]}
(x), \hat u^1_{[0,r^1-k^1]})$
and~\eqref{eq:flat_output_deriv_proof}:

\emph{Item b)ii):} It follows that the components in
directions $\D \hat u^2_{[1,r^j-p^j]}$, $j\in\{2,3\}$,
can only cancel within $\Spanf{\D\varphi_{[R]}}$ if
$r^2-p^2 = r^3-p^3$.
Otherwise,~\eqref{eq:imp_exterior_deriv} would be
contradicted.

\textit{Item b)iii):} By eliminating all terms involving
\mbox{$\D \hat u^1_{[0,r^1-k^1-1]}$} and
\mbox{$\D \hat u^2_{[0,r^2-p^2-1]}$} in
\mbox{$\D \varphi_{[0,R-1]}$} one obtains a total of
\mbox{$k^1 + p^2 + r^3$} one-forms contained in
$\Spanf{\D x}$. Hence,
\mbox{$n < k^1 + p^2 + r^3$} would imply linear
dependence, contradicting flatness. Conversely,
\mbox{$n > k^1 + p^2 + r^3$} would imply
$\Spanf{\D x} \not\subset
\Spanf{\D \varphi_{[0,R-1]}}$. Consequently, the state
dimension is given by
\mbox{$n = k^1 + p^2 + r^3$}. From
\mbox{$\ddiff{\varphi}=\idxsum{R}-n$},~%
$d_{\mathrm{diff}}(\varphi) = (r^1-k^1)+(r^2-p^2)$ follows directly.

\emph{Item b)i):} Let us take the exterior derivative
of~\eqref{eq:deriv_form_case_rank1}. First, assume that
$r^1-k^1 < r^2-k^2$. If additionally
\mbox{$r^2-k^2 \neq r^3-k^3$}, say
$r^2-k^2 < r^3-k^3$, then $p^2-k^2 < p^3-k^3$ 
because of item \emph{b)ii)}.
As a result, $\D \varphi^3_{[r^3]}$ would
contain components in the directions
$\D \hat u^1_{[r^2-k^2+1, r^3-k^3]}$ that cannot be
canceled by any linear combination within
$\Spanf{\D\varphi_{[R]}}$,
contradicting~\eqref{eq:imp_exterior_deriv}. The case
$r^2-k^2 > r^3-k^3$ is ruled out analogously. Thus
$r^1-k^1 < r^2-k^2$ implies $r^2-k^2 = r^3-k^3$. Hence,
among the differences $r^i-k^i$, only
one can be strictly smaller while the two largest
coincide.

\emph{Item c):} According to item \textit{b)i)} of
Theorem~\ref{thm:rank1_case},
assumption~\eqref{eq:rearrangement_case_rank1} is always
valid. Next, consider an invertible input transformation
of the form~\eqref{eq:input_trf_case_rank1}.

\emph{Case i):} If $r^1 - k^1 = \ddiff{\varphi}$, then
from item \textit{b)}, \mbox{$r^2 - p^2 = r^3 - p^3 = 0$},
meaning $p^2 = r^2$ and $p^3 = r^3$. Thus, $u^2$ and
$u^3$ first appear at orders $r^2$ and $r^3$,
respectively. The
structure~\eqref{eq:deriv_form_case_rank1} then shows that
the map \mbox{$(x, \hat u^1_{[0,r^1-k^1]}, \hat u^2,
\hat u^3) \mapsto \varphi_{[0,R]}$} is a diffeomorphism,
as the number of variables on both sides coincides and
$\D \varphi_{[0,R]}$ are linearly independent. Hence,
extending the original system by $d^1 = r^1 - k^1$
prolongations of $\hat u^1$ renders the system
\mbox{$\varphi$-SFL} via minimal prolongations with
$D = (r^1 - k^1)$ for any invertible
transformation~\eqref{eq:input_trf_case_rank1}.

\emph{Case ii):} If $r^1 - k^1 < \ddiff{\varphi}$, then
\mbox{$r^2 - p^2 = r^3 - p^3 > 0$}. For the system to be
\mbox{$\varphi$-SFL} via minimal prolongations with an
input transformation of the
form~\eqref{eq:input_trf_case_rank1}, there must exist a
transformation
$\hat u^2 = \phi_{\hat u^2}(x, \hat u^1, u^2, u^3)$ such
that $\varphi^2_{[p^2,r^2-1]}$ and
$\varphi^3_{[p^3,r^3-1]}$ can be expressed entirely in
terms of $x$, $\hat u^1_{[0,r^1-k^1-1]}$, and
$\hat u^2_{[0,r^2-p^2-1]}$. That means, after such a
transformation, the integer $s$
in~\eqref{eq:deriv_form_case_rank1} is given by
$s = r^2-p^2$. If such $\hat u^2$ exists, the map
$(x, \hat u^1_{[0,r^1-k^1]},
\hat u^2_{[0,r^2-p^2]}, \hat u^3)
\mapsto \varphi_{[0,R]}$ is a diffeomorphism, and
$D = (r^1-k^1, r^2-p^2)$. Conversely, if no such
$\hat u^2$ exists, then the dependence of $u^2,u^3$ and
their derivatives in $\varphi_{[0,R-1]}$ cannot be
captured by a single prolonged input, and the system is
not \mbox{$\varphi$-SFL} any~\eqref{eq:input_trf_case_rank1}.
\end{proof}


\end{document}